\documentclass{article}

\usepackage{amsmath} % 
\usepackage{amssymb} % 
\usepackage{amstext}
\usepackage{amsthm}
\usepackage[utf8]{inputenc}
\usepackage{fullpage}
\let\phi=\varphi
\def\eps{\varepsilon}

\newtheorem{theorem}{Theorem} % 1st argument is your name for it
\newtheorem{lemma}[theorem]{Lemma}     % 2nd argument is what is printed
\newtheorem{corollary}[theorem]{Corollary}
\newtheorem{proposition}[theorem]{Proposition}

\newtheorem{remark}[theorem]{Remark}
\newtheorem*{remark*}{Remark}

\newtheorem{definition}[theorem]{Definition} % 1st argument is your name for it

\long\def\koment#1{}

\newcommand{\Arg}{\operatorname{Arg}}
\newcommand{\nas}{nonaxisymmetric}
\newcommand{\tf}{\tilde f}
\newcommand{\talpha}{\tilde \alpha}
\newcommand{\tha}{\tilde h}

\renewcommand{\Im}{\operatorname{Im}}
\renewcommand{\Re}{\operatorname{Re}}

\newcommand{\R}{\mathbb{R}}
\newcommand{\LL}{\mathcal{L}}
\newcommand{\C}{\mathbb{C}}

\newcommand{\U}{\mathcal{U}}
\newcommand{\V}{\mathcal{V}}
\newcommand{\A}{\mathcal{A}}
\newcommand{\HH}{\mathcal{H}}
\newcommand{\X}{\mathcal{X}}

\newcommand{\F}{\mathcal{F}}

\newcommand{\T}{\mathcal{T}}
\newcommand{\VV}{\mathcal {G}}
\newcommand{\cV}{\mathcal {D}_\sigma}

\newcommand{\dvg}{\operatorname{div}}
\def\ra{\right\rangle}
\def\la{\left\langle}
\newcommand{\dual}[2]{\la#1,#2\ra}  

\def\L2s{L^2_{\sigma}(\Omega)}
\newcommand{\Lds}{L^2_{\nu}(\partial\Omega)}
\def\Ldva{L^2(\Omega)}
\def\H2s{H^2_{\sigma}(\Omega)}
\def\Hjs{H^1_{\sigma}(\Omega)}
\def\Hpulsg{H^{1/2}_{\nu}(\partial\Omega)}
\newcommand{\Htripulsg}{H^{3/2}_{\nu}(\partial\Omega)}
\def\Hpulsgm{H^{\frac12}_{\nu}(\partial\Omega)}
\newcommand{\Htripulsgm}{H^{\frac32}_{\nu}(\partial\Omega)}

\newcommand{\bneg}{\beta^{-1}}

\title{Maximal regularity of Stokes problem with dynamic boundary condition --- Hilbert setting}

\author{{Tomáš} {Bárta}, {Paige} {Davis}, {Petr} {Kaplický}} %%%\email{barta@karlin.mff.cuni.cz}
\date{}

\begin{document}

\maketitle

\begin{abstract}
{
For the evolutionary Stokes problem with dynamic boundary conditions, we show the maximal regularity of weak solutions in time. Due to the characterization of $R$-sectorial operators on Hilbert spaces, the proof reduces to identifying the appropriate functional analytic setting and proving that the corresponding operator is sectorial, i.e., that it generates an analytic semigroup.
}
\end{abstract}

\section{Introduction}

%\par Modelling
Certain materials, like polymer melts, can slip over solid surfaces. Such boundary behavior is described by slip velocity models; see \cite[Section~6]{Ha2012} for an overview.  Moreover, it has been observed that the slip is often not constant but varies over time, depending on the fluid's current state. Such fluids need to be represented using dynamic slip models. They were first proposed in \cite{PePe1968} in a general form
$$
u_\tau+\lambda_\tau\partial_t u_\tau=\varphi(\sigma_w),
$$
where $u_\tau$ is the slip velocity, $t$ stands for the time, $\lambda_\tau$ is the slip relaxation time, $\sigma_w$ stands for the wall shear stress and $\varphi$ should be determined based on the rheological properties of the fluid under consideration.

%\par previous studies of dynamic boundary condition in fluid mechanics
The mathematical studies of problems with dynamic boundary conditions in the context of fluid mechanics started by the thesis of Maringová, \cite{maringova2019}. She studied the existence of solutions to systems of (Navier)-Stokes type under various constitutive relations for the extra stress tensor and the modified dynamic boundary condition $s(u_\tau)+\partial_t u_\tau=-\sigma_w$ with a given--possibly nonlinear--function $s$. These results were later published in \cite{AbBuMa2021}. 

%\par our motivation
We are interested in the optimal regularity of problems with dynamic boundary conditions in the context of Lebesgue spaces. Specifically, we focus on the linear Stokes problem. First, we find the result interesting. Second, it provides a basis for studying the regularity of more complex systems. Moreover, the linear theory can be considered as a tool for the reconstruction of pressure; see \cite{SovonWa1986}.

% \par setting
We study the problem
\begin{align}	\label{P1.1}
\partial_t u - \Delta u + \nabla p &= f \qquad \text{ in } I\times\Omega,
\\		\label{P1.2}
\dvg u &=0 \qquad \text{ in } I\times\Omega,
\\		\label{P1.3}%\tag{P1.3}
\beta\partial_t u + (2Du\cdot \nu)_{\tau} + \alpha u_{\tau}&= \beta g  \qquad \text{ in } I\times\partial\Omega,
\\	\label{P1.4}
u_{\nu} &= 0 \qquad \text{ in } I\times\partial\Omega,
\\		\label{P1.5}
u &= u_0 \qquad \text{ in } \{0\}\times\Omega \\   \label{P1.6}
u &= v_0 \qquad \text{ in }  \{0\}\times\partial\Omega
\end{align}
in a bounded domain $\Omega\subset \mathbb R^d$, $d\geq2$ with $C^{2,1}$ boundary and a time interval $I=(0,T)$, $T>0$. 
The constants $\alpha\in\R$, $\beta>0$, the functions $f:I\times\Omega\to \mathbb R^d$, $g:I\times\partial\Omega\to \mathbb R^d$, $u_0:\Omega\to \mathbb R^d $ and $v_0:\partial\Omega\to \mathbb R^d $ are given. Subscripts $(\cdot)_\tau$ and $(\cdot)_\nu$ denote the tangential and the normal part of the vectors. We look for unknown functions $u:I\times\Omega\to\mathbb R^d$ and $p:I\times\Omega\to\mathbb R$. Let us mention that we permit $\alpha<0$, however, only $\alpha\ge 0$ seems to be physically relevant.

The notion of the weak solution is adopted (with a small modification) from \cite[Section~5]{maringova2019}. We work in Banach spaces
$$
\VV=\{(u,u_b)\in H^1_{\sigma}(\Omega)\times \Lds: u_b=\gamma(u)\},\quad
\HH=L^2_{\sigma}(\Omega)\times \Lds
$$
with norms
$$
\|(u,u_b)\|_\VV^2=2\|Du\|^2_{\Ldva} + \|u_b\|^2_{L^2(\partial\Omega)},\quad
\|(u,u_b)\|_\HH^2=\|u\|^2_{\Ldva} + \beta\|u_b\|^2_{L^2(\partial\Omega)}.
$$
Definitions of all mentioned function spaces can be found in Subsection~\ref{sec:not-fs}. Note, that $\VV$ is a dense subset of $\HH$. 

The duality pairing between $\VV$ and its dual space $\VV^*$, denoted $\dual{\cdot}{\cdot}_{\VV}$, extends the scalar product in $\HH$; see \cite[Section~3.1]{maringova2019}.

When dealing with a function from $\VV$ %or $\VV\cap\HH$ 
we write only the first component of the vector. The trace of the function is automatically considered as the second component.

\begin{definition}\label{def:weak-p} Let $0<T\leq+\infty$, $\alpha\in \R$, $\beta>0$, $\Omega\subset\R^d$, $\Omega\in C^{0,1}$, $f \in L^1_{loc}([0,T), \Hjs^*)$, $g\in L^1_{loc}([0,T), \Lds)$, $u_0\in \L2s$ and $v_0\in \Lds$. We say that $u$ is a weak solution to the problem \eqref{P1.1}--\eqref{P1.6} if $u\in L^2_{loc}([0,T), \VV)\cap C_{loc}([0,T), \HH)\cap L^\infty_{loc}([0,T),\HH)$, $\partial_t u\in L^1_{loc}(0,T,\VV^*)$, $u(0)=(u_0,v_0)$ in $\HH$, and the equations \eqref{P1.1} and \eqref{P1.3} are satisfied in the weak sense, i.e.,
\begin{equation}\label{eq:weak1}
  \dual{\partial_t u}{\varphi}_{\VV}+2\int_\Omega Du:D\varphi + \alpha\int_{\partial\Omega} u\varphi=
  \dual{(f,g)}{\varphi}_{\VV}
\end{equation}
almost everywhere on $(0, T)$ and for all $\varphi\in \VV$.
\end{definition}

% \par result
Note that $\beta$ is hidden in \eqref{eq:weak1} in the definition of $\dual{\cdot}{\cdot}_{\VV}$.
Under the regularity assumptions in Definition~\ref{def:weak-p}, equality \eqref{P1.3} makes no sense when understood pointwise. Instead, this part of the boundary condition is hidden in the weak formulation \eqref{eq:weak1}. It can be derived pointwise only if the regularity of the data and solution is better; see Theorem~\ref{thm:main}. 

We are interested in the maximal regularity of weak solutions with respect to the problem data, i.e., the right hand side functions $f$ and $g$,  and the initial values $u_0$ and $v_0$. In order to state the precise conditions for the initial values we need to introduce spaces
$$
\begin{gathered}
  \X_0 = \L2s \times \Hpulsg, \quad \X_1=\{(u,u_b)\in \H2s \times \Hpulsg: \gamma(u)=u_b\}, \\
  \|(f,g)\|_{\X_0}=\|{f}\|_{L^2(\Omega)}+\|{g}\|_{H^{1/2}(\partial\Omega)},\quad \|{(u,v)}\|_{\X_1}=\|{u}\|_{H^2(\Omega)}+\|{v}\|_{H^{1/2}(\partial\Omega)},\\
  \X_{1-\frac1q,q}=(\X_0,\X_1)_{1-\frac1q,q}.
\end{gathered}
$$
The last space is the real interpolation space between $\X_0$ and $\X_1$. It turns out that this is the optimal (largest possible) space for the initial data to guarantee $L^q$-maximal regularity of solutions to the non-homogeneous abstract Cauchy problem; see \cite[Section~2.2.1]{Lunardi1995} for details.

Before we formulate the main theorem, we need some preparation for \nas\ domains; this Lemma is a consequence of Lemma~\ref{lem:korn} below. 

\begin{lemma}\label{lem:prekorn}
Let $\Omega$ be \nas. There exists $\alpha_0<0$ such that for all $u\in H^1(\Omega)$ with $u\cdot\nu=0$ on $\partial\Omega$
$$
2\alpha_0\|u\|^2_{L^2(\partial\Omega)}+4\|Du\|^2_{L^2(\Omega)} \geq 0.
$$

\end{lemma}
From this point forward, $\alpha_0$ always refers to the fixed constant from Lemma~\ref{lem:prekorn}.

Our main theorem follows. 

\begin{theorem}\label{thm:main}
  Let one of the following conditions be met:
  \begin{enumerate}%[(a)]
    \renewcommand{\theenumi}{\alph{enumi}}
  \item $T\in(0,+\infty)$,\label{ass:thma}
  \item $T=+\infty$, $\alpha>0$,\label{ass:thmb}
  \item $T=+\infty$, $\alpha\in(\alpha_0,0]$, $\Omega$ \nas.\label{ass:thmc}
  \end{enumerate}
  For every $q\in (1,+\infty)$, there exists a constant $C>0$ such that, for every $\F=(f,g)\in L^q(I, \X_0)$ and $(u_0, v_0)\in \X_{1-1/q,q}$ there exists a weak solution $u$ of \eqref{P1.1}--\eqref{P1.6}, it is unique, and satisfies $u\in L^q(I,H^2(\Omega))$, $\partial_tu\in L^q(I,L^2(\Omega))$. Moreover, there exists a function $p\in L^q(I, H^1(\Omega))$ such that \eqref{P1.1}--\eqref{P1.6} hold pointwise almost everywhere and
\begin{multline}\label{final_reg}
\|\partial_tu(t)\|_{L^q(I,L^2(\Omega))} +\|u(t)\|_{L^q(I,H^2(\Omega))} + \|p(t)\|_{L^q(I,H^1(\Omega))} \le C(\|\F\|_{L^q(I, \X_0)}+\|(u_0,v_0)\|_{\X_{1-1/q,q}}).
\end{multline}
\end{theorem}

% original version of the theorem for mild solutions
% \begin{theorem}\label{thm:main}
% For every $q\in (1,+\infty)$ there exists $C>0$ such that for every $\F=(f,g)\in L^q(I, \X_0)$ and $u_0\in \H2s$, $v_0=\gamma(u_0)$ the unique mild solution $(u, p)\in L^q(I,H^2(\Omega)\times H^1(\Omega))$ of \eqref{P1.1}--\eqref{P1.6} satisfies
% \begin{equation}\label{final_reg}
% \|u(t)\|_{L^q(I,H^2(\Omega))} + \|\pi(t)\|_{L^q(I,H^1(\Omega))}\le C(\|\F\|_{L^q(I, \X_0)}+\|u_0\|_{H^2}).
% \end{equation}
% \end{theorem}

%\subsection{Motivation}
Our approach to the problem is as follows. We rewrite the problem \eqref{P1.1}--\eqref{P1.4} as an abstract Cauchy problem 
\begin{equation}\label{ACP}
\partial_t \U = \A \U + \F(t),  
\end{equation}
on a Hilbert space $\X_0$. Since the problem combines evolutionary equations in the interior of $\Omega$ and on its boundary, the space $\X_0$ must be a product of spaces in the interior and on the boundary of $\Omega$, compare \cite{Esch1992, DePruZa2008}.
We show below that $\A$ is the generator of an analytic semigroup $\T$. Then the Variation-Of-Constants-Formula
\begin{equation}\label{Sg_solution}
\U(t)%=\binom{u(t)}{u_b(t)}
=\T(t)\U_0 + \int_0^t \T(t-s)\F(s) ds
\end{equation}
defines a mild solution to \eqref{ACP}. This mild solution actually has better properties if $\U_0$ and $\F$ are sufficiently good.
% (we make the assumptions on $\F$ more precise below, for now let us consider $\F=0$).
Namely, since $\X_0$ is a Hilbert space we have maximal $L^p$ regularity, i.e., for every $\F\in L^p(I,\X_0)$ the solution $\U$ given by \eqref{Sg_solution} with $\U_0=0$ satisfies $\A\U$, $\dot \U\in L^p(I,\X_0)$ and 
$$
\|\dot \U\|_{L^p(I,\X_0)}+\|\A\U\|_{L^p(I,\X_0)}\le C\|\F\|_{L^p(I,\X_0)}
$$
with a constant $C>0$ independent of $\F$, compare \cite{deSi1964} or \linebreak\cite[Corollary~1.7]{KuWe2004}. Since the mild solution is very regular we show that it is actually a weak solution from Definition~\ref{def:weak-p}. Uniqueness of the weak solution then concludes the argumentation. 

% This in particular means that and \eqref{P1.2}, \eqref{P1.4} hold for all $t>0$ and since $\U(t)\in D(\A)$ for $t>0$ we have $u_b(t)=\gamma(u(t))$. Moreover, since \eqref{ACP} holds we have \eqref{P1.3} and $\partial_t u = P\Delta u + f$. We find $p$ such that $\nabla p = \Delta u -P\Delta u$. Then \eqref{P1.1} holds and we have a solution to problem (P).

% \par Wentzel boundary condition
% \par previous studies without pressure

Apart from articles \cite{maringova2019} and \cite{AbBuMa2021} we are aware only of the work \cite{PraZe2023}, that appeared recently. In this article its authors study whether the solutions to \eqref{P1.1}--\eqref{P1.6} are given by an analytic semigroup in spaces $L^p_\sigma(\Omega)\times L^p_\nu(\partial\Omega)$ with $p>1$. The result is rather involved but does not cover our result since we work in $\X_0=L^2_\sigma(\Omega)\times H^{1/2}_\nu(\partial\Omega)$. A variant of dynamic boundary conditions appeared also in a different context. In \cite{Ven1959} they appeared as general boundary conditions that turn a given elliptic differential operator to the generator of a semigroup of positive contraction operators. There are many works on dynamic boundary conditions (or Wentzell\footnote{Note that Wentzell and Ventcel' are different spelling of the same name. The first form is used in MathSciNet, the second one in literature.} boundary conditions) in the context of parabolic and hyperbolic equations without the incompressibility constraint and without the pressure. Our main example are the results in \cite{DePruZa2008} where the maximal $L^p$ regularity is proved for a very general class of parabolic systems equipped with a general dynamic boundary condition. The presented article can be considered as the first step to a parallel theory for the Stokes problem.

In the following section we give the basic notation and define the operator $\A$. Elliptic theory is studied in Section~\ref{sec:el-reg}. The proof of Theorem~\ref{thm:main} is given in Section~\ref{sec:pf-mt}.

\section{Notation and functional analytic setting}
\subsection{Notation and function spaces}\label{sec:not-fs}

The constant $\alpha_0<0$ is a fixed constant from Lemma~\ref{lem:prekorn}. If $z:\R^d\to\R^d$ then $(\nabla z)_{ij}=\partial_jz_i$ and $(Dz)_{ij}=\frac12(\partial_jz_i+\partial_iz_j)$ for $i,j\in\{1,\dots,d\}$. If $A,B$ are matrices, then $AB$ denotes the matrix product, e.g., $([\nabla z]z)_{i}=\partial_kz_iz_k$ for $i\in\{1,\dots,d\}$, while $A:B=a_{ij}b_{ij}$. We use the summation convention over repeated indices. For two vectors $a,b\in\R^d$, $a\cdot b$ denotes the scalar product in $\R^d$.

We recall that $\Omega\subset\R^d$ is a bounded domain with $C^{2,1}$ boundary, $I=(0,T)$ for some $T>0$.

The standard Sobolev and Sobolev-Slobodeckii spaces over $\Omega$ with integrability $2$ and differentiability $s>0$ are denoted by $H^s$. Further,
$$
\begin{gathered}
  \cV=\{u\in C^\infty_0(\Omega); \dvg u =0\},\quad
\L2s = \mbox{closure of $\cV$ in $L^2(\Omega)$}, \\
\Hjs = H^1\cap \L2s,\quad\H2s  = H^2\cap \L2s.
\end{gathered}
$$
We write $\gamma$ for the trace operator.
If $w$ is a function defined on $\Omega$ with trace $\gamma(w)$ on $\partial\Omega$ we denote $w_\nu$ its normal part and $w_\tau$ its tangential part  on $\partial\Omega$. By $\nu(x)$ we denote the unit outer normal vector to $\partial \Omega$ at point $x\in\partial\Omega$. Equalities of functions are understood almost everywhere with respect to the corresponding Hausdorff measure.

We remark that if $w\in \Hjs$  %\cap \H2s=\H2s$,
 then $\dvg w=0$ in $\Omega$ in the weak sense and $w_\nu=0$ on $\partial\Omega$. Consequently, if we define
$$
\begin{gathered}
\Lds =\{w\in L^2(\partial\Omega): w_\nu=0\quad\mbox{a.e. on $\partial\Omega$}\},\\ % vzhledem k Hausdorfove mire
\Hpulsgm = \{w\in H^{\frac12}(\partial\Omega): w_\nu=0\mbox{ on $\partial\Omega$}\},\\
\Htripulsgm=\{w\in H^{\frac32}(\partial\Omega); w_\nu=0\mbox{ on $\partial\Omega$}\},
\end{gathered}
$$
then $\Hpulsg=\gamma(\Hjs)$ and $\Htripulsg=\gamma(\H2s)$.
% \footnote{
% % \superset clear
% % \subset
%   The inclusion $\subset$ follows by inverse trace theorem, Bogovskii Theorem and \cite[Theorem~2.9]{AmGi1994}.}
%we need that $\{w\in H^1, \dvg w=0, w_n=0\}\subset\L2s$. It should easily for starshaped domains by shinking and mollification. The inner boundary then causes no problems due to $w_\nu=0$. Find reference.

The Helmholtz-Weyl decomposition
% ??? original reference ???,
yields $L^2(\Omega)=G_2(\Omega)\otimes \L2s$ where
$$
G_2(\Omega)=\{w\in L^2(\Omega);w=\nabla p, p\in H^1\},
$$
see, e.g., \cite[Theorem III.1.1]{Galdi2011}.
The continuous Leray projection of $L^2(\Omega)$ to $\L2s$ is denoted $P:L^2(\Omega)\to\L2s$.

 \subsection{Definition of the operator $\A$}
The operator $\A$ is considered on the space $\X_0$.
The domain of $\A$ is defined as $D(\A) = \X_1$.
Finally, we set 
\begin{equation}
\quad\A \binom{u}{u_b} = \binom{P \Delta u}{-\bneg[(2Du \cdot \nu)_{\tau} + \alpha u_b]}\quad \mbox{for } \binom{u}{u_b} \in D(\A).
\end{equation}

\section{Regularity theory for the elliptic problem}\label{sec:el-reg}
Before proving that $(\A, D(\A))$ generates an analytic semigroup in $\X_0$, we establish some preliminary results on the existence and regularity of solutions to the following system.
%\begin{equation}\label{eq:1}-\label{eq:bdr2}
\begin{align}
  \lambda u - \Delta u + \nabla \pi &= f \qquad \text{ in } \Omega,
  \label{eq:1}
\\		
\dvg u &=0 \qquad \text{ in } \Omega,
\label{eq:2}
\\		
\lambda u_{\tau} + \bneg[(2Du\cdot \nu)_{\tau} + \alpha u_{\tau}] &=  h  \qquad \text{ in } \partial\Omega,
\label{eq:bdr1}
\\		
u_{\nu} &= 0 \qquad \text{ in } \partial\Omega.
\label{eq:bdr2}
\end{align}
%\end{equation}
Since the operator $\A$ is defined on a product space, we retain this structure also in this part of the presentation. However, this is not strictly necessary, because the second component of the space is the trace of the first.

In this part we work in the space $\VV$. We recall its definition
$$
\VV=\{(u,u_b)\in H^1_{\sigma}(\Omega)\times \Lds: u_b=\gamma(u)\}
$$
with norm
$$
\|(u,u_b)\|^2_\VV=2\|Du\|^2_{\Ldva} + \|u_b\|^2_{L^2(\partial\Omega)}.
$$
The norm in $\VV$ is equivalent to the norm in $H^1(\Omega)$ due to Korn's and Poincar\'e's inequalities, originally established in \cite{NeHla1970a}; see also \cite[Proposition~3.13]{AceAmConGho2021}.
For reader's convenience, we present it here.  The essential part of Lemma~\ref{lem:korn} is taken from \cite[Proposition~3.13]{AceAmConGho2021}. The last equivalence of norms is the standard Korn's inequality.

\begin{lemma}\label{lem:korn}
Let $\Omega$ be a bounded Lipschitz domain. Then, for all $u\in H^1(\Omega)$ with $u_\nu=0$ on $\partial\Omega$, we have
$$
\|u\|_{H^1(\Omega)}\sim\|Du\|_{L^2(\Omega)}
$$
if $\Omega$ is \nas,
and
$$
\|u\|_{H^1(\Omega)}\sim\|Du\|_{L^2(\Omega)}+\|u_\tau\|_{L^2(\partial\Omega)}
\sim\|Du\|_{L^2(\Omega)}+\|u\|_{L^2(\Omega)}
$$
if $\Omega$ is arbitrary.
% More generally, if  is axisymmetric and α ∈ L 2 () satisfies (2.1), then the following equiva-
% lence holds
Here, ``$\sim$'' denotes the equivalence of two norms.
\end{lemma}
The first part of this lemma proves Lemma~\ref{lem:prekorn}. 
We continue with the definition of a weak solution to \eqref{eq:1}--\eqref{eq:bdr2}.
\begin{definition}\label{def:weak}
Let $(f,h)\in L^2(\Omega)\times L^2(\partial\Omega)$ (complex valued) and let $\alpha\in\R$,  $\lambda\in \C$.
We say that $(u,u_b)\in \VV$ is a weak solution to \eqref{eq:1}--\eqref{eq:bdr2} if 
\begin{equation} \label{weaksol}
 \lambda \int_{\Omega} u \bar \phi + \int_{\Omega} 2Du:\nabla \bar \phi + (\beta\lambda+\alpha) \int_{\partial \Omega} u_b \bar \phi_b = \int_{\Omega} f\bar \phi + \int_{\partial\Omega} \beta h\bar \phi_b
\end{equation}
holds for every $(\phi,\phi_b)\in \VV$.
\end{definition}
Note again that the boundary condition \eqref{eq:bdr1} is embedded within the weak formulation \eqref{weaksol} and cannot be expressed pointwise under regularity assumptions of Definition~\ref{def:weak}. However, if $u$ is more regular, e.g., $u\in H^2(\Omega)\cap\Hjs$, then one can show that \eqref{eq:1}--\eqref{eq:bdr2} hold pointwise almost everywhere in $\Omega$ or $\partial\Omega$; see Proposition~\ref{strongsol}. 

We will use the standard definition of the sector
\begin{definition}
  For $\omega\in\R$, $\theta\in(0,\pi)$ we define
  $$
  S_{\theta,\omega}=\{\lambda\in\C;\lambda\neq\omega,|\arg(\lambda-\omega)|<\theta\}.
  $$
Throughout this article, $\arg$ denotes the continuous branch of the argument function, taking values in $[-\pi,\pi)$.
\end{definition}

We aim to prove results on the existence, uniqueness and estimates of the weak solutions. Before formulating these results we need some preparatory lemmata. We start with a simple lemma on properties of complex numbers.

\begin{lemma} \label{lem:cn1}
Let $\theta\in(0,\pi)$ and $\Arg$ be a continuous branch of argument.
Then 
\begin{equation}\label{eq:c1}
|a\lambda  + b\mu|\geq \cos(\theta/2)(a|\lambda|+b|\mu|)
\end{equation}
for all $a, b>0$ and all $\lambda,\mu\in\C\setminus\{0\}$ such that $|\Arg(\lambda)-\Arg(\mu)|\leq\theta$.
In particular,
\begin{equation}\label{eq:c2}
|a\lambda  + b|\geq \cos(\theta/2)(a|\lambda|+b)
\end{equation}
for all $a, b >0$ and all $\lambda \in \overline{S_{\theta,0}}$.
\end{lemma}
\begin{proof}
  To prove \eqref{eq:c2}, we realize that due to the fact that $b>0$ we can separately treat the cases when $\Im\lambda>0$ and $\Im\lambda\leq 0$. Then it is sufficient to set $\Arg=\arg$ and apply \eqref{eq:c1} for $\lambda\ne 0$, whereas for $\lambda=0$ is \eqref{eq:c2} obvious.

To prove \eqref{eq:c1}, we set $\omega=(\Arg\lambda+\Arg\mu)/2$ and $\gamma=\theta/2$. Then we have 
\begin{equation} \label{e:args}
|\arg(e^{-i\omega}a\lambda )|, |\arg(e^{-i\omega}b\mu )| \le\gamma,
\end{equation}
i.e., $e^{-i\omega}a\lambda$ and $e^{-i\omega}b\mu$ belong to $\overline{S_{\gamma,0}}$. Obviously, for any $z\in \overline{S_{\gamma,0}}$ we have $\Re z \ge |z|\cos\gamma$.
Now, we can estimate
\begin{equation*}
\begin{aligned}
|a\lambda + b\mu|&=|e^{-i\omega}(a\lambda + b\mu)|\ge \Re(e^{-i\omega}(a\lambda + b\mu)) 
= \Re(e^{-i\omega}a\lambda )  + \Re(e^{-i\omega}b\mu )
\ge \cos\gamma\left(a|\lambda|+b |\mu|\right).
\end{aligned}
\end{equation*}
\end{proof}

The next lemma deals with a fundamental estimate needed for proof of existence of the weak solutions and also for spectral estimates.

\begin{lemma} \label{lem:complexnumbers}
  Let $\alpha\in\R$, $\beta>0$, and let $\omega\in\R$ be such that
  \begin{equation}\label{eq:poincare}
    \exists C>0, \forall u\in\VV:
    2\|Du\|^2_{\Ldva} + (\alpha+\beta\omega) \|u_b\|^2_{L^2(\partial\Omega)}+\omega \|u\|^2_{\Ldva} \ge C\|u\|^2_{\VV}. 
  \end{equation}
  
  Then for every $\theta\in(0,\pi)$ there exists $c>0$ such that for all $\lambda\in \overline{S_{\theta,\omega}}$ and $\U=(u,u_b)\in \VV$ the following inequality holds
	\begin{equation}\label{est:coerc}
	\left|\lambda \|u\|^2_{\Ldva} + 2\|Du\|^2_{\Ldva} + (\alpha+\beta \lambda) \|u_b\|^2_{L^2(\partial\Omega)}\right| \ge c \|\U\|^2_{\VV}+ c|\lambda-\omega| \|\U\|^2_{\HH}.
\end{equation}
\end{lemma}

\begin{proof}
The estimate \eqref{est:coerc} clearly holds for $u=(0,0)\in\VV$. Take $\lambda\in \overline{S_{\theta,\omega}}$ and $u\in\VV\setminus\{(0,0)\}$ arbitrary. Denote 
$$
M(\lambda)=\lambda \|u\|^2_{\Ldva} + 2\|Du\|^2_{\Ldva} + (\alpha+\beta \lambda) \|u_b\|^2_{L^2(\partial\Omega)}.
$$ 
Then
\begin{equation} \label{e:rozlozeno1}
M(\lambda)= (\lambda-\omega)\left(\|u\|^2_{\Ldva} + \beta\|u_b\|^2_{L^2(\partial\Omega)}\right) + 2\|Du\|^2_{\Ldva} + (\alpha+\beta\omega)\|u_b\|^2_{L^2(\partial\Omega)}+\omega\|u\|^2_{\Ldva}.
\end{equation}
Let us observe that $2\|Du\|^2_{\Ldva} + (\alpha+\beta\omega)\|u_b\|^2_{L^2(\partial\Omega)}+\omega\|u\|^2_{\Ldva}>0$ by \eqref{eq:poincare} and $\lambda-\omega\in \overline{S_{\theta,0}}$. 
Equality \eqref{e:rozlozeno1} together with Lemma~\ref{lem:cn1} and \eqref{eq:poincare} imply
\begin{equation*}
|M(\lambda)|\geq c \|\U\|^2_{\VV}+ c|\lambda-\omega| \|\U\|^2_{\HH}.
\end{equation*}
\end{proof}

\begin{remark}\label{rem:poinc}
  The condition \eqref{eq:poincare} is a version of the Poincar\'e-Korn inequality. Note that it is valid if any of the following conditions is met
  \begin{enumerate}
    \renewcommand{\theenumi}{\alph{enumi}}
  \item  $\alpha>0$ and $\omega\ge 0$, \label{c2}
	\item $\Omega$ \nas { }, $\alpha\in(\alpha_0,0]$ and $\omega\ge 0$,\label{c3}
	\item $\alpha\le 0$ and $\omega>-\alpha/\beta$. \label{c1}
\end{enumerate}  
Indeed, this follows for the cases \ref{c2} and \ref{c1} directly from Lemma~\ref{lem:korn}. In the case \ref{c3} one also needs to exploit Lemma~\ref{lem:prekorn}.
\end{remark}

In the next proposition we prove the existence and uniqueness of weak solutions in the set~$\VV$.
\begin{proposition} \label{propweaksol}
  Let  %\L2s\times \Hpulsg$
  $\alpha\in\R$, $\beta>0$, $\theta\in(0,\pi)$, and $\omega\geq0$ be such that one of the conditions \ref{c2}, \ref{c3}, \ref{c1} of Remark~\ref{rem:poinc} be satisfied. 
  Then there exists $C>0$ such that for all $(f,h)\in L^2(\Omega)\times L^2(\partial\Omega)$, $\lambda\in \overline{S_{\theta,\omega}}$ there exists a unique weak (complex-valued) solution $(u,u_b)\in \VV$ of \eqref{eq:1}--\eqref{eq:bdr2} satisfying $u_b\in \Hpulsg$. 
	Further, there exists a unique $\pi\in L^2(\Omega)$ with $\int_{\Omega}\pi=0$ such that 
\begin{equation} \label{weaksolp}
\int_{\Omega} f\bar \phi + \int_{\partial\Omega} \beta h\bar \phi_b = \lambda \int_{\Omega} u \bar \phi + \int_{\Omega} 2Du:\nabla \bar \phi-\int_{\Omega} \pi\dvg \bar \phi + (\beta\lambda+\alpha) \int_{\partial \Omega} u_b \bar \phi_b
\end{equation}
for all $\phi\in H^1(\Omega)$. The following estimate holds
\begin{equation}\label{est:weaksol}
  \|u\|_{\VV}+|\lambda-\omega|\|u\|_{\HH}+\|\pi\|_{L^2(\Omega)}\leq C(\|f\|_{L^2(\Omega)}+\|h\|_{L^2(\partial\Omega)}).
\end{equation}
\end{proposition}

\begin{proof}
We define the sesquilinear form 
$$
B(\U,\V)=\lambda \int_{\Omega} u \bar v + \int_{\Omega} 2Du:\nabla \bar v + (\alpha+\beta\lambda) \int_{\partial \Omega} u_b \bar v_b
$$
on $\VV\times \VV$ where $\U=(u,u_b)$, $\V=(v,v_b)$. Lemma~\ref{lem:complexnumbers} and Remark~\ref{rem:poinc} imply the existence of $C>0$ independent of $\U$ and $\lambda$ such that 
%\begin{equation*}
\begin{align*}
|B(\U,\U)| &= \left|\lambda \|u\|^2_{\Ldva} + 2\|Du\|^2_{\Ldva} + (\alpha+\beta\lambda) \|u_b\|^2_{L^2(\partial\Omega)}\right| \ge C \|\U\|^2_\VV
\end{align*}
%\end{equation*}
Moreover, the form $B$ is bounded from above on $\VV$. 

By the Lax-Milgram theorem (see, e.g., \cite{Pet1965}), for $\F\in \VV^*$ defined by $\F(\Phi)=\int_{\Omega} f\bar \phi + \int_{\partial\Omega} \beta h\bar \phi_b$ for $\Phi\in \VV$ there exists a unique $\U=(u,u_b)\in \VV$ such that $B(\Phi,\U)=\F(\Phi)$ for every $\Phi=(\phi,\phi_b)\in \VV$, i.e., \eqref{weaksol} holds. By the trace theorem, $u_b=\gamma(u)\in \Hpulsg$. Estimate \eqref{est:weaksol} of $u$ follows from $B(\U,\U)=\F(\U)$, properties of $B$ and $\F$, Lemma~\ref{lem:complexnumbers} and Remark~\ref{rem:poinc}. 

Let us now prove existence and uniqueness of $\pi$. By \cite[Theorem III.5.3]{Galdi2011} any weak solution $u$ defined in Definition~\ref{def:weak} can be associated with a pressure $\pi\in L^2(\Omega)$ satisfying
\begin{equation} \label{weaksolp-pre}
\int_{\Omega} f\bar \phi = \lambda \int_{\Omega} u \bar \phi + \int_{\Omega} 2Du:\nabla \bar \phi-\int_{\Omega} \pi\dvg \bar \phi
\end{equation}
for any $\phi\in H^1(\Omega)$ with $\gamma(\phi)=0$. The pressure is defined uniquely up to an additive constant. Let us further require the constant to be chosen in such a way that the pressure has zero mean over $\Omega$. Then the validity of the estimate \eqref{est:weaksol} for pressure follows from \cite[Lemma~IV.1.1]{Galdi2011} and the estimate \eqref{est:weaksol} for $u$. 

For $\phi\in H^1$ such that $\phi_\nu=0$ on $\partial\Omega$ we can find 
$z\in H^1(\Omega)$ such that $\gamma(z)=0$ on $\partial\Omega$ and $\dvg z = \dvg \phi$ in $\Omega$; see \cite[Theorem III.3.1]{Galdi2011}. Now, $z$ is an admissible test function in \eqref{weaksolp-pre} and $\phi-z\in H^1_{\sigma}(\Omega)$ is an admissible test function in \eqref{weaksol}. Subtracting the so obtained equalities one gets 
\eqref{weaksolp} for all $\phi\in H^1$ such that $\phi_\nu=0$ on $\partial\Omega$ 

%$\psi\in H^1_{\sigma}(\Omega)$ such that $\gamma(\psi)=\gamma(\phi)$ on $\partial\Omega$, see \cite[Theorem III.3.1]{Galdi2011}. Finally, $\phi-\psi$ is an admissible test function in \eqref{weaksolp-pre} and $\psi$ is an admissible test function in \eqref{weaksol}. Subtracting the so obtained equalities one gets 
%\eqref{weaksolp} for all $\phi\in H^1$ such that $\gamma(\phi)_\nu=0$ on $\partial\Omega$ 
\end{proof}

\begin{remark}\label{rem:linup}
Note, that the mapping $(f,h)\in \L2s\times\Lds \mapsto (u,\pi)\in \VV\times L^2(\Omega)$ from the previous theorem is linear and bounded. 

  Since the parametres $\alpha\in\R$ and $\beta>0$ are fixed, we do not track the dependence of the constant $C$ on these parameters in Proposition~\ref{propweaksol} and also in all further estimates.
\end{remark}

% \begin{lemma}\label{lem:bogovski-reg}
%   There is a continuous linear mapping $B:H^1_0(\Omega)\to H^2(\Omega)$ such that for all $f\in H^1_0(\Omega)$ we have $\dvg (Bf)=f$ in $\Omega$.
% \end{lemma}
% \begin{proof}
%   Find it somewhere.
% \end{proof}

Before we state our result on regularity of weak solutions we need to prove a lemma on existence of a special function satisfying boundary conditions.

\begin{lemma}\label{lem:extention}
There exists $C>0$ such that for every $h\in \Hpulsg$ there exists $w\in H^2(\Omega)$ with properties 1) $\dvg(w)=0$ in $\Omega$, 2) $\gamma(w)=0$ on $\partial\Omega$, 3) $(2 \gamma(Dw)\cdot\nu)_\tau=h$ on $\partial\Omega$ and 4) $\|w\|_{H^2(\Omega)}\leq C\| h\|_{H^{1/2}(\partial\Omega)}$. 
\end{lemma}

\begin{remark}
  Regularity of $w\in H^2(\Omega)$ together with 1) and 2) imply $w\in\H2s$.
\end{remark}

\begin{proof}[Proof of Lemma~\ref{lem:extention}]
  Step 1: We construct a function $z\in H^2(\Omega)$ satisfying conditions 2)-4) and, additionally, 5) $\dvg z=0$ on $\partial\Omega$. By the inverse trace theorem (see, e.g., \cite[Theorem~2.5.8]{Ne2012})
  % \footnote{Lions Magenes require $\partial\Omega\in C^\infty$. Less regularity of $\partial\Omega$ is probably needed in the book of Necas \cite{Ne2012} but I do not have its copy.},
  there exists $z\in H^2(\Omega)$ such that $\gamma(z)=0$ and $\gamma(\nabla z)\nu= h$ on $\partial\Omega$ and $\|z\|_{H^2(\Omega)}\leq C\|h\|_{H^{1/2}(\partial\Omega)}$. This function $z$ obviously satisfies 2) and 4). Since $\gamma(z)= 0$ on $\partial \Omega$, it follows that $\gamma(\nabla z)\xi=0$ for any tangent vector $\xi$ to $\partial \Omega$. Consequently, 
$$
[\gamma(\nabla z)^T\nu]\cdot \xi =\xi^T\gamma(\nabla z)^T\nu = [\gamma(\nabla z)\xi]^T\nu=0\nu=0.
$$ 
Thus, we obtain $(2[\gamma(Dz)]\nu)_\tau=([\gamma(\nabla z)]\nu)_\tau+([\gamma(\nabla z)]^T\nu)_\tau=([\gamma(\nabla z)]\nu)_\tau=h_{\tau}=h$, which confirms 3). Further, Héron's formula
%\footnote{$\partial\Omega\in C^{1,1}$ viz AmGi1994, Remark 3.6}
(see \cite[Lemme 3.3]{He1981} or \cite[Lemma 3.5]{AmGi1994}) yields 
\begin{equation}\label{eq:heron}
  \dvg z=\dvg_{\partial\Omega}(z_\tau)+[\gamma(\nabla z)]\nu\cdot\nu- 2K z_\nu\quad\mbox{on $\partial\Omega$.} 
\end{equation}
In the formula, $K$ denotes the mean curvature of $\partial\Omega$ and $\dvg_{\partial\Omega}$ denotes the surface divergence. All three terms on the right-hand side of \eqref{eq:heron} are zero since $\gamma(z)=0$ on $\partial \Omega$ and $[\gamma(\nabla z)]\nu\cdot \nu = h\cdot \nu=0$. So, 5) holds.

Step 2: It remains to correct the solenoidality of $z$ without destroying the conditions $2)-4)$. To do this we apply \cite[Theorem 2]{Bog1980}
% \footnote{It would be nice to formulate the statement by operator. I did not find it yet. Something is in Galdi's book}
to the problem $\dvg \zeta = \dvg z$ in $\Omega$. Since $\dvg z\in H^1_0(\Omega)$ and $\int_{\Omega}\dvg z=\int_{\partial\Omega}z_\nu=0$ there exists a solution $\zeta\in H^2_0(\Omega)$ of this problem such that $\|\zeta\|_{H^2}\leq C\|\dvg z\|_{H^1}\leq C\|z\|_{H^2}\leq C\|h\|_{H^{1/2}}$.

Finally, it remains to define $w=z-\zeta$. This function satisfies all conditions 1)--4).  
\end{proof}

\begin{theorem} \label{T_reg}
Under the assumptions of Proposition~\ref{propweaksol} the unique weak solution $(u,u_b)$ of \eqref{eq:1}--\eqref{eq:bdr2} and the associated pressure $\pi$ satisfy 
$(u,u_b)\in D(\A)$, $\pi\in H^1(\Omega)$ for all $f\in L^2(\Omega)$, $h\in \Hpulsg$. Moreover, there exists $C>0$ independent of $\lambda$, $f$, $h$ such that
\begin{equation}\label{est:h2}
  \|u_b\|_{H^{3/2}(\partial\Omega)}+\|u\|_{H^2(\Omega)} + \|\pi\|_{H^1(\Omega)}
  \le C(\|f\|_{L^2(\Omega)}+\|h\|_{H^{1/2}(\partial\Omega)}).
\end{equation}

% The associated pressure $\pi$ satisfies $\pi\in H^1(\Omega)$ and
% $$
% \|\pi\|_{H^1(\Omega)} %+ \|\pi\|_{H^1(\Omega)}
%   \le C(\Omega)(\|f\|_{L^2(\Omega)}+\|h\|_{H^{1/2}(\partial\Omega)}).
% $$

\end{theorem}

\begin{proof}
According to the definition of $D(\A)$ it suffices to show $u\in H^2(\Omega)$, $\pi\in H^1(\Omega)$ together with the estimate \eqref{est:h2}. Note that the estimate of the boundary value $u_b$ follows from the estimate of $u$ in $H^2(\Omega)$ by the trace theorem.

%There exists a weak solution $u$ to \eqref{eq:1}-\eqref{eq:bdr2} by Proposition \ref{propweaksol}.

For $\lambda\in\R$ we rewrite the system in the form
$$%\begin{equation}\label{eq:weak1}
\begin{gathered}
  - \Delta u + \nabla \pi = f-\lambda u, \qquad \dvg u =0,  \qquad \text{ in } \Omega,
\\		
(2Du\cdot \nu)_{\tau} +(\beta\lambda+\eta+\alpha) u_{\tau} =  \beta h +\eta u_{\tau},\qquad u_{\nu} = 0  \qquad \text{ in } \partial\Omega,
\end{gathered}
$$%\end{equation}
where $\eta=|\alpha|+1$.
Any of assumptions \ref{c2}--\ref{c1} of Remark~\ref{rem:poinc} implies $(\beta\lambda+\eta+\alpha)\geq1$. We have $\|f-\lambda u\|_{L^2(\Omega)}\leq C(\|f\|_{L^2(\Omega)}+\|h\|_{L^{2}(\partial\Omega)})$ and $\|\beta h +\eta u_{\tau}\|_{H^{1/2}(\partial\Omega)}\leq C(\|f\|_{L^2(\Omega)}+\|h\|_{H^{1/2}(\partial\Omega)})$ by Proposition~\ref{propweaksol}. Therefore, we can apply \cite[Theorem 4.5]{AceAmConGho2021} to get the estimate \eqref{est:h2}.

If $\lambda\in\C$ we still have a weak solution $u$ of \eqref{eq:1}--\eqref{eq:bdr2} by Proposition \ref{propweaksol}. We would like to apply a complex valued analogue of \cite[Theorem 4.5]{AceAmConGho2021} to
\begin{equation}\label{eq:weak2}
\begin{gathered}
  - \Delta v + \nabla \sigma = \tf, \qquad \dvg v =0,  \qquad \text{ in } \Omega,
\\		
(2Dv\cdot \nu)_{\tau} +\talpha v_{\tau} =  \tha,\qquad v_{\nu} = 0  \qquad \text{ in } \partial\Omega,
\end{gathered}
\end{equation}
where $\tf=f-\lambda u\in L^2(\Omega)$, $\talpha=\beta\lambda+\eta+\alpha\in\overline{S_{\theta,1}}$ and $\tha=\beta h +\eta u_{\tau}\in \Hpulsg$ with norms independent of $\lambda$. The proof presented in \cite{AceAmConGho2021} works also in the complex valued situation with minor changes.

As in that article, we can again assume without loss of generality that $\tha=0$. In fact, if $\tha\ne 0$ we consider a solenoidal function $w\in \H2s$ satisfying the equation \eqref{eq:weak2}$_2$ on the boundary.  Such a function exists and is independent of $\talpha$ due to Lemma~\ref{lem:extention} and satisfies $\|w\|_{H^2(\Omega)}\le C\|\tha\|_{H^{1/2}(\partial\Omega)}$. Then it suffices to study the solution to \eqref{eq:weak2} with the right hand side $\tilde f  + \Delta w\in L^2(\Omega)$ and $\tha=0$.

To show the regularity of a weak solution to \eqref{eq:weak2} with $\tha=0$ and of the associated pressure we apply the method of difference quotients as in \cite{AceAmConGho2021}. It can be followed almost line by line. The only difference  is in obtaining regularity at the boundary in the tangent direction since our parameter $\talpha$ is complex. We present here the main idea of this estimate in the case that we deal with the flat portion of the boundary. 
%The detailed explanation of the method in the case of non-flat boundary is in \cite{Bei2004}. 
Let $x_0\in\partial\Omega$, $r>0$, $U:=B(x_0,r)$, $2U:=B(x_0,2r)$ be such that $\partial\Omega\cap 2U$ is a subset of a hyperplane perpendicular to $e_d$.  We test the weak formulation of the equation \eqref{eq:weak2} by the complex conjugate $\overline{D^{-h}_k(\zeta^2D^h_k v)}$, where $\zeta\in{\cal D}(2U)$, $\zeta\geq\chi_{U}$ is a cut-off function that localizes our consideration to the neighborhood of the flat boundary and $D^{h}_k$ is a difference quotient of size $h\ne 0$ taken in the direction $e_k$ parallel to the boundary, i.e., $D^h_k v(x)=(v(x+he_k)-v(x))/h$ for $x\in\R^d$. We can follow the computation in the section (i) of the proof of \cite[Theorem~4.5]{AceAmConGho2021} almost line by line to get
$$
2 \int_{\Omega} \zeta^2|D^h_kD v|^2 +\talpha \int_{\partial\Omega} \zeta^2 |D^h_k v_{\tau}|^2\leq C(\|f\|_{L^2(\Omega)}^2+\|\pi\|_{L^2(\Omega)}^2+\|\nabla u\|_{L^2(\Omega)}^2).
$$
Here comes the only difference in the argumentation, since to estimate the left hand side from below we need to employ Lemma~\ref{lem:cn1} with $\talpha-1\in \overline{S_{\theta,0}}$
\begin{multline*}
\left|2 \int_{\Omega} \zeta^2|D^h_kD v|^2 +\talpha \int_{\partial\Omega} \zeta^2 |D^h_k v_{\tau}|^2\right| = \left|2 \int_{\Omega} \zeta^2|D^h_kD v|^2 +\int_{\partial\Omega} \zeta^2 |D^h_k v_{\tau}|^2+(\talpha-1) \int_{\partial\Omega} \zeta^2 |D^h_k v_{\tau}|^2\right| \\
\ge c\left(2 \int_{\Omega} \zeta^2|D^h_kD v|^2 +\int_{\partial\Omega} \zeta^2 |D^h_k v_{\tau}|^2\right).
\end{multline*}
Hence, one can continue as in \cite{AceAmConGho2021} to conclude that solutions of \eqref{eq:weak2} satisfy
$$
\|v\|_{H^2(U)} + \|\sigma\|_{H^1(U)} \le C(\|\tf\|_{L^2(\Omega)}+\|\tha\|_{H^{1/2}(\partial\Omega)})
$$
which implies \eqref{est:h2} by the flat portion of $\partial\Omega$. The full estimate \eqref{est:h2} is obtained by localization and flattening the boundary. For details see \cite{Bei2004}.

% Under the regularity of $u$ and $\pi$ from the theorem we can apply the Divergence theorem to the second and the third terms on the right hand side of \eqref{weaksolp} and obtain for all $\phi\in H^1(\Omega)$, $\phi_b=\gamma(\phi)$ with $(\phi_b)_\nu=0$ on $\partial\Omega$
% $$
% \int_{\Omega} (\lambda u - \Delta u +\nabla\pi-f) \bar \phi + \int_{\partial\Omega} [(\beta\lambda+\alpha) u_b + 2Du \nu - h]\bar \phi_b = 0
% $$
% which yields that $u$ satisfies \eqref{eq:1} pointwisely almost everywhere in $\Omega$  and \eqref{eq:bdr1} pointwisely almost everywhere in $\partial\Omega$. The values of $u$, $Du$ on boundary are understood in the sense of traces. Note that from the second integral we only get information on the tangent part of the functions since $(\varphi_b)_\nu=0$ on $\partial\Omega$.
\end{proof}

\begin{remark}\label{rem:pressure-operator}
  It follows from Theorem~\ref{T_reg}, Proposition \ref{propweaksol} and Remark~\ref{rem:linup} that the mapping associating the pressure with zero mean to the problem data, $(f,h)\in \L2s\times\Hpulsg\mapsto \pi\in H^1(\Omega)$, is linear and bounded from $\L2s\times\Hpulsg$ to $H^1(\Omega)$.
\end{remark}  

\begin{remark}
It seems to us that in \cite{AceAmConGho2021} the result corresponding to the previous Theorem is announced for bounded domains $\Omega$ with $C^{1,1}$ boundary. As the main reference for the technique that allows to get the result in the neighborhood of the nonflat boundary is presented \cite{Bei2004}. We are not able to reconstruct the proof for $C^{1,1}$ domains and we want to remark that also in \cite{SolSca1973} and \cite{Bei2004} it is assumed that the boundary of $\Omega$ is  $C^3$ and $C^{2,1}$ respectively. 

The proof of regularity up to the boundary is done in the following steps. In order to avoid troubles with nonflat boundary, the problem is reformulated as a regularity problem with flat boundary and finally for this problem the technique of differences is used to show regularity of its solutions.

\newcommand{\tu}{\tilde u}
Let us discuss the flattening of the boundary in more detail. We assume that $x_0\in\partial\Omega$, $r>0$ and $\Omega\cap B(x_0,r)=\{x\in B(x_0,r); x_d>H(x_1,\dots,x_{d-1})\}$, where $H:\R^{d-1}\to\R$ is a given function parametrizing the boundary of $\Omega$. Moreover, the coordinate system corresponding to $x_0$ is chosen in such a way that $H(0,\dots,0)=0$, $\nabla'H(0,\dots,0)=0$. A modified solution is defined by the formula
$\tu(x',x_d-H(x')):=(u'(x),u_d(x)-\nabla'H(x')\cdot u'(x))$ for $x\in \Omega\cap B(x_0,r)$, where $x'=(x_1,\dots,x_{d-1})$, $u'=(u_1,\dots,u_{d-1})$, $\nabla'=(\partial_1,\dots,\partial_{d-1})$. Note that it is defined on a subset of $\{x\in\R^d;x_d>0\}$. The term $\nabla'H(x')\cdot u'(x)$ is subtracted from the last component of $u$ to enforce $\dvg \tu =0$ in new coordinates.

It can be shown that the function $\tu$ then again solves a variant of the Stokes problem and that this function is $H^2(B((x_1,\dots,x_{d-1},0),\rho)\cap \{x\in\R^d;x_d>0\})$ for some $\rho>0$. One should reconstruct from this fact that also the original function $u$ is in $H^2$. However, the term $\nabla'H(x')\cdot u'(x)$ stands in the way. One needs to show that it also belongs to $H^2$. Here one uses the choice of the coordinate system. It is clear that some information on third derivative of $H$ is needed, e.g., $H\in C^{2,1}$. In \cite[Section 4, page 1096]{Bei2004} a variant $H\in W^{3,3}$ is also discussed. %We think that $H\in C^{1,1}$ is not sufficient.
\end{remark}

\begin{proposition} \label{strongsol}
Under the assumptions of Proposition~\ref{propweaksol} let $\F=(f,h)\in \L2s\times \Hpulsg$. The weak solution $\U:=(u,u_b)\in \VV$ of \eqref{eq:1}--\eqref{eq:bdr2}  belongs to $D(\A)$ and satisfies $\lambda \U-\A\U=\F$.
  %, then there exists $\pi\in H^1$ such that $(u,\pi)$ satisfies \eqref{eq:1}-\eqref{eq:bdr2}.
\end{proposition}

\begin{proof}
Since $(u,u_b)\in D(\A)$ and the associated pressure $\pi\in H^1$ by Theorem~\ref{T_reg}, we have  $u\in H^2$  and by \eqref{weaksolp}
\begin{equation} \label{seven}
\begin{aligned}
\int_{\Omega} f\bar \phi + \int_{\partial\Omega} \beta h\bar \phi_b &= \lambda \int_{\Omega} u \bar \phi + 2\int_{\Omega}Du:\nabla \bar \phi -\int_{\Omega} \pi\dvg \bar \phi + (\beta\lambda+\alpha) \int_{\partial \Omega} u_b \bar \phi_b\\
&= \lambda \int_{\Omega} u \bar \phi +2\int_{\partial\Omega} [Du]\nu \bar \phi_b - \int_{\Omega} \Delta u \bar \phi +\int_{\Omega} \nabla\pi \bar \phi+ (\beta\lambda+\alpha) \int_{\partial \Omega} u_b \bar \phi_b\\
\end{aligned}
\end{equation}
for any $\phi\in H^1$, $\phi_b=\gamma(\phi)$ with $(\phi_b)_\nu=0$ on $\partial\Omega$. It follows $\lambda u-\Delta u+\nabla\pi=f$ a.e. in $\Omega$ which gives $\lambda u - P(\Delta u) = f$. Inserting the pointwise equality $\lambda u-\Delta u+\nabla\pi=f$  into \eqref{seven} for a general $\phi\in H^1$, $\phi_b=\gamma(\phi)$ with $(\phi_b)_\nu=0$ on $\partial\Omega$ we obtain
$$
\int_{\partial\Omega} \beta h\bar \phi_b = 2\int_{\partial\Omega} [Du]\nu \bar \phi_b + (\beta\lambda+\alpha) \int_{\partial \Omega} u_b \bar \phi_b.
$$
Due to the regularity of $(u,u_b)$ we get $\beta h =  2([Du]\nu)_\tau  + (\beta\lambda+\alpha)  u_b$ a.e. on $\partial\Omega$.
\end{proof}

% \begin{lemma} \label{lemmaequiv}
%   Under the assumptions of Proposition~\ref{propweaksol} let $\F=(f,h)\in \L2s\times \Hpulsg$ and $u\in \H2s$. The couple $(u,\gamma(u))$ solves $(\lambda I - \A)\U = \F$ if and only if 
%   there exists $\pi\in H^1$ with zero mean over $\Omega$ such that $(u,\pi)$ solves \eqref{eq:1}-\eqref{eq:bdr2}.

%   Moreover, the mapping $u\in D(\A)\subset \H2s \mapsto \pi\in H^1$ is linear and bounded.
% \end{lemma}

% \begin{proof}
% First assume that $\U=(u,\gamma(u))\in D(\A)$ solves $(\lambda I - \A)\U = \F$. Then $\lambda u - P\Delta u = f$ and there exists $\pi\in H^1$ such that $\nabla \pi=\Delta u - P\Delta u$, so first and second equations in \eqref{eq:1}-\eqref{eq:bdr2} hold. The third and fourth equations of \eqref{eq:1}-\eqref{eq:bdr2} hold obviously. On the contrary, let $(u,\pi)\in \H2s \times H^1$ solves \eqref{eq:1}-\eqref{eq:bdr2}. By the second and fourth equations and the trace theorem $(u,\gamma(u))\in D(\A)$. By the first equation $P\Delta u= P(\lambda u + \nabla \pi - f)=\lambda u -f$.
% \end{proof}

\section{Proof of the main theorem}\label{sec:pf-mt}
\subsection{Uniqueness of the weak solutions}
Our Definition~\ref{def:weak-p} of the weak solution differs from the one in \cite[Definition~5.1]{maringova2019} in the assumption on regularity of the right-hand side function $(f,g)$ and of the solution $\partial_tu$. That is why we present here a simple proof of the uniqueness of the weak solutions.

\begin{lemma}\label{lem:uniq}
  In the situation of Definition~\ref{def:weak-p}, let $u,v$ be two weak solutions corresponding to the same data $f$, $g$, $u_0$, $v_0$. Then $u=v$.
\end{lemma}
\begin{proof}
  We define $w=u-v$. Then $w$ is a weak solution corresponding to the trivial data. In particular, it solves \eqref{eq:weak1} with zero right hand side. From this equation we read that actually $\partial_tw\in L^2_{loc}([0,T),\VV^*)$ and consequently $w$ is the unique weak solution on any $(0,T^*)$ with $T^*\in(0,T)$ in the spirit of \cite[Definition~5.1]{maringova2019}, for uniqueness see \cite[Theorem~5.1]{maringova2019}. It follows that $w=0$ and $u=v$.
\end{proof}

\subsection{The operator $\A$ generates an analytic semigroup}

We show that $(\A, D(\A))$ is densely defined, closed, its resolvent set contains a sector and resolvent estimates are satisfied there; see \eqref{est:resolv}. We start with

\begin{proposition} \label{P_dense}
Let $\alpha\in \R$, $\beta>0$. $D(\A)$ is dense in $\X_0$ and $(\A,D(\A))$ is a closed operator.
\end{proposition}

\begin{proof}
  We first prove density of $D(\A)$ in $\X_0$. Let $(f,h)\in\X_0$ and $\eps>0$. Due to the density of $H^{3/2}(\partial\Omega)$ in $H^{1/2}(\partial\Omega)$ there exists $\tilde h_1\in H^{3/2}(\partial\Omega)$ such that  $\|h-\tilde h_1\|_{H^{1/2}(\partial\Omega)}<\eps$. Then we  orthonormally project $\tilde h_1$ to the tangent bundle of $\partial\Omega$ and denote the resulting function $h_1$. Since $\Omega$ has $C^{2,1}$ boundary, the ortonormal projection does not spoil the regularity of $h$. Indeed, $h_1$ can be written as $h_1(x)=\tilde h_1(x)-\langle\tilde h_1(x),\nu(x)\rangle \nu(x)$. 
%  ??? What is precise required regularity of boundary - guess is $\partial\Omega \in H^{5/2}$- $C^{2,1}$ suffices?
Consequently, $h_1\in \Htripulsg$. Moreover, since $h_{\nu}=0$ we have $\|h-h_1\|_{H^{1/2}(\partial\Omega)}\le \|h-\tilde h_1\|_{H^{1/2}(\partial\Omega)}<\eps$.
We now find $f_1\in \H2s$ such that $\gamma(f_1)=h_1$. Its existence follows from \cite[Corollary 3.8]{AmGi1994}. Finally, by definition of $\L2s$ there exists $f_2\in \cV$ such that $\|(f-f_1)-f_2\|_{\L2s}<\eps$. Then $(f_1+f_2,h_1)\in D(\A)$ is the desired approximation of $(f,h)\in\X_0$.

To show closedness of $\A$ let $\U=(u,b)$, $\U_n=(u_n,b_n)\in D(\A)$ be such that $\U_n\to\U$ in $\X_0$ and $\F_n:=\A\U_n\to \F$ in $\X_0$ as $n\to+\infty$. In particular,
%\begin{equation} \label{conv}
\begin{align*}
  \gamma(u_n)=b_n,\quad
u_n \to u  \quad\text{ in $\L2s$}, \quad
b_n \to b  \quad\text{ in $\Hpulsg$ as $n\to+\infty$.} \\
%\A \U_n \to \V &  \quad\text{ in $\X_0$}
\end{align*}
%\end{equation}
Since $\{G_n\}=\{\lambda\U_n-\F_n\}$ is a bounded sequence in $\X_0$ we can apply Theorem~\ref{T_reg} to the equation $\lambda\U_n-\A\U_n=G_n$ with $\lambda=\max(0,-\alpha/\beta)+1$. We get that the sequence $\{\|u_n\|_{H^2}\}$ is bounded, so a subsequence $\{v_n\}$ of $\{u_n\}$ converges weakly in $H^2(\Omega)$ to some $v$.  By the convergence $u_n \to u$ in $\L2s$ we have $v=u$ and necessarily $u\in H^2(\Omega)$.
% Since the solutions to \eqref{eq:1}-\eqref{eq:bdr2} are unique, we actually obtain that the whole sequence $u_n\rightharpoonup u$ in $H^2$.
Due to the continuity of the trace mapping, the embeddings and the Leray projection $P$ we also get $\gamma(v_n)\rightharpoonup \gamma(u)$ in $H^{3/2}(\partial\Omega)$, $\dvg v_n\rightharpoonup \dvg u$ in $H^1(\Omega)$, and $P\Delta v_n \rightharpoonup P\Delta u$ in $\L2s$. Thus, we conclude that $b=\gamma(u)$, $u\in \H2s$, $(u,b)\in D(\A)$, and $\A\U=\F$.
\end{proof}

\begin{theorem} \label{t:sectoriality}
Let $\alpha\in \R$, $\beta>0$. The operator $(\A, D(\A))$ is sectorial. More precisely, if $\omega\in\R$ is such that $\alpha$, $\beta$ and $\omega$ satisfy one of the conditions in Remark~\ref{rem:poinc} then for any $\theta\in(0,\pi)$ there exists $C>0$ such that 
  \begin{equation}\label{est:resolv}
  \overline{S_{\theta,\omega}}\subset\rho(\A),\quad \forall\lambda\in \overline{S_{\theta,\omega}}:|\lambda-\omega|\|(\lambda-\A)^{-1}\|\leq C.
\end{equation}
If 1) $\alpha>0$ or 2) $\alpha>\alpha_0$ and $\Omega$ \nas, then $\omega$ can be chosen negative.
\end{theorem}

\begin{proof}
  Let $\theta\in(0,\pi)$ and $\lambda\in \overline{S_{\theta,\omega}}$.
  By Proposition \ref{strongsol}, for every $\F=(f,h)\in \X_0$ there exists a solution $\U=(u,u_b)\in D(\A)$ satisfying $(\lambda-\A)\U=\F$, i.e., the operator $\lambda-\A:D(\A)\to\X_0$ is surjective. Since any solution to $(\lambda-\A)\U=\F$ corresponds to a unique weak solution of \eqref{eq:1}--\eqref{eq:bdr2} (by Proposition~\ref{propweaksol}), $\lambda-\A$ is also injective. The operator $\lambda-\A$ is closed by Proposition~\ref{P_dense} and we obtain $\lambda \in \rho(\A)$. In particular, $(\lambda-\A)^{-1}$ is bounded. 

% density of domain of $A$ is needed for strong continuity of the analytic semigroup, see Lunardi1995, p. 34. For analyticity/sectoriality it is not necessary.

%   By Proposition \ref{P_dense}, $D(\A)$ is dense in $\X_0$. We show that for any $\lambda_0>0$
% there exists $M>0$ such that all $\lambda\in \C$, $\Re\lambda\ge \lambda_0$ belong to $\rho(\A)$ and 
% $$
% \|(\lambda I -\A)^{-1}\|\le \frac{M}{|\lambda|}
% $$
% holds.
% , i.e., $\lambda\in \rho(\A)$. 

Next, we establish the resolvent estimate, i.e., the inequality
$
|\lambda-\omega|\|\U\|_{\X_0}\le C \|\F\|_{\X_0}
$.
This can be reformulated for $\U=(u,u_b)$ and $\F=(f,h)$ as
$$
|\lambda-\omega|\left(
\|u\|_{L^2(\Omega)} + \|u_b\|_{H^{1/2}(\partial\Omega)}\right)\le C(\|f\|_{L^2(\Omega)}+\|h\|_{H^{1/2}(\partial\Omega)}).
$$
In Proposition~\ref{propweaksol} we have already proved
\begin{equation}\label{est:r1}
|\lambda-\omega|
\|u\|_{L^2(\Omega)}\le C(\|f\|_{L^2(\Omega)}+\|h\|_{H^{1/2}(\partial\Omega)}).
\end{equation}
From the second component of the equation $(\lambda-\A)\U=\F$ (see \eqref{eq:bdr1}), we have
$$
|\lambda|\|u_b\|_{H^{1/2}(\partial\Omega)}\le |\alpha|\| u\|_{H^{1/2}(\partial\Omega)} + C\|D u\|_{H^{1/2}(\partial\Omega)} + \|h\|_{H^{1/2}(\partial\Omega)},
$$
and therefore
\begin{equation}\label{est:r2}
|\lambda -\omega|\|u_b\|_{H^{1/2}(\partial\Omega)}\leq
(|\lambda |+|\omega|)\|u_b\|_{H^{1/2}(\partial\Omega)}\le (|\alpha|+|\omega|)\| u\|_{H^{1/2}(\partial\Omega)} + C\|D u\|_{H^{1/2}(\partial\Omega)} + \|h\|_{H^{1/2}(\partial\Omega)}).
\end{equation}
We estimate the terms containing $u$ on the right-hand side by the trace theorem and by Theorem \ref{T_reg} as
$$
\| u\|_{H^{1/2}(\partial\Omega)} + \|D u\|_{H^{1/2}(\partial\Omega)}\le  C \|u\|_{H^2(\Omega)} \le C(\|f\|_{L^2(\Omega)}+\|h\|_{H^{1/2}(\partial\Omega)})
$$
to get \eqref{est:resolv} combining \eqref{est:r1} and \eqref{est:r2}.

If moreover 1) or 2) holds, then the statement is already established with $\omega=0$. Moreover, $0\in \rho(A)$, implying that a neighborhood of zero belongs to $\rho(A)$. Fix $\omega<0$ within this neighborhood. Then, for appropriate $\theta'<\theta$ we have $S_{\theta',\omega}\subset \rho(A)$, and the resolvent estimates hold on this sector (by standard arguments). As $\theta'\to \pi$ when $\theta\to\pi$, it follows that the resolvent estimates hold on $S_{\theta,\omega}$ for each $\theta<\pi$, with $C$ depending on $\theta$.
\end{proof}

\begin{corollary}\label{cor:sem-bdd}
  The operator $(\A, D(\A))$ generates an analytic semigroup $\{\T(t)\}_{t>0}\subset\LL(\X_0)$. There exist constants $\omega\in \R$ and $C>0$ such that the semigroup satisfies for any $t>0$
\begin{equation}\label{est:sem-bdd}
  \|\T(t)\|_{\LL(\X_0)}\leq C e^{\omega t}.
\end{equation}
If moreover 1) $\alpha>0$ or 2) $\alpha>\alpha_0$ and $\Omega$ \nas, $\omega$ can be chosen negative. 
\end{corollary}
\begin{proof}
  The statement follows directly from \cite[Proposition~2.1.1]{Lunardi1995} and Theorem \ref{t:sectoriality}.
\end{proof}

Now we are ready to prove the main result, Theorem~\ref{thm:main}.

% However, we first reformulate the result in order to have optimal regularity of the initial condition. Let us denote $\Y_q=(\X_0,D(\A))_{(1-1/q,q)}$ the real interpolation space between between $\X_0$ and $D(\A)$, see e.g. \cite{Lunardi1995}[Definition 1.2.2] for definition. Then we have the following theorem. Theorem ~\ref{thm:main} is its corollary.

% \begin{theorem}\label{thm:main2}
% For every $q\in (1,+\infty)$ there exists $C>0$ such that for every $\F=(f,g)\in L^q(I, \X_0)$ and $U_0=(u_0,v_0)\in \Y_q$ the unique mild solution $(u, p)\in L^q(I,H^2(\Omega)\times H^1(\Omega))$ of \eqref{P1.1}--\eqref{P1.6} satisfies
% \begin{equation}\label{final_reg}
% \|u(t)\|_{L^q(I,H^2(\Omega))} + \|\pi(t)\|_{L^q(I,H^1(\Omega))}\le C(\|\F\|_{L^q(I, \X_0)}+\|U_0\|_{Y_q}).
% \end{equation}
% \end{theorem}

\subsection{Proof of Theorem~\ref{thm:main}}

%\begin{proof}[Proof of Theorem~\ref{thm:main}]
We begin by proving the regularity of the weak solution and the estimate \eqref{final_reg} under assumptions \ref{ass:thmb} or \ref{ass:thmc}. Let one of them hold, in particular $I=(0,+\infty)$. Then $\A$ is densely defined, closed and generates a bounded analytic semigroup $\T$ on the Hilbert space $\X_0$ by Proposition~\ref{P_dense} and Corollary~\ref{cor:sem-bdd}. We moreover have the estimate \eqref{est:sem-bdd} with $\omega<0$ at our disposal. From \cite[Theorem~1.1, Corollary~1.7 and (1.9)]{KuWe2004} the operator $\A$ has maximal $L^q$ regularity, i.e., the mild solution $\U^0$ of \eqref{ACP} (see \eqref{Sg_solution}) with $\U(0)=0$ satisfies $\dot\U^0$, $\A\U^0\in L^q(I, \X_0)$ and 
$$
\|\dot\U^0\|_{L^q(I, \X_0)} + \|\A\U^0\|_{L^q(I, \X_0)}\le C\|\F\|_{L^q(I, \X_0)}.
$$
Let us denote $\U^1(t):=\T(t)\U_0$ for $t>0$ the mild solution of \eqref{ACP} with $\F=0$ and $\U(0)=\U_0$. Since $\U_0\in \X_{1-1/q,q}$, the function $t\mapsto\A\U^1(t)$ (and therefore also $t\mapsto \dot\U^1(t)$) belongs to $L^q((0,1),\X_0)$ and the inequality
$$
\|\dot\U^1\|_{L^q((0,1), \X_0)} + \|\A\U^1\|_{L^q((0,1), \X_0)}\le C\|\U_0\|_{\X_{1-1/q,q}}
$$
holds (see \cite[Proposition 2.2.2 and formula (2.2.3)]{Lunardi1995}). From the properties of analytic semigroups (see \cite[Proposition~2.1.1]{Lunardi1995}), we get for $t>0$
$$
\|\A\U^1(t)\|_{\X_0}=\frac1t\|t\A\T(t)\U_0\|_{\X_0}\leq\frac {Ce^{\omega t}}t\|\U_0\|_{\X_0}\le \frac{C e^{\omega t}}t\|\U_0\|_{\X_{1-1/q,q}}.
$$
Since $\kappa(t):=e^{\omega t}/t$ satisfies $\kappa\in L^q(1,+\infty)$ it follows that
$$
\|\dot\U^1\|_{L^q(I, \X_0)} + \|\A\U^1\|_{L^q(I, \X_0)}\le C\|\U_0\|_{\X_{1-1/q,q}}.
$$
Hence, the solution $\U=(u,u_b)=\U^0+\U^1$ of \eqref{ACP} satisfies $\dot\U$, $\A\U\in L^q(I, \X_0)$, and 
$$
\|\dot\U\|_{L^q(I, \X_0)} + \|\A\U\|_{L^q(I, \X_0)}\le C(\|\F\|_{L^q(I, \X_0)}+\|\U_0\|_{\X_{1-1/q,q}}).
$$

For a.e. $t>0$ % ??? je to potreba dokazovat?
$$
-\A\U(t) = \F(t) - \dot\U(t)\quad\mbox{in $\X_0$}
$$
and $\U(t)$ is also the unique weak solution to \eqref{eq:1}--\eqref{eq:bdr2} with $\lambda=0$ and the right hand side $(\tilde f,\tilde h)=(f(t)-\partial_tu(t), h(t)-\partial_t u_b)\in \X_0$.
It follows from Theorem~\ref{T_reg}, with assumptions \ref{c2} or \ref{c3} of Remark~\ref{rem:poinc}, that the functions $u(t)$, $u_b(t)$ and the associated pressure $\pi(t)$ satisfy the estimate
\begin{multline}\label{est:final}
  \|u_b(t)\|_{H^{\frac32}(\partial\Omega)}+\|u(t)\|_{H^2(\Omega)} + \|\pi(t)\|_{H^1(\Omega)}
\\  \le C(\|f(t)\|_{L^2(\Omega)}+\|\partial_t u(t)\|_{L^2(\Omega)}+\|h(t)\|_{H^{1/2}(\partial\Omega)}+\|\partial_tu_b(t)\|_{H^{1/2}(\partial\Omega)}).
\end{multline}
Since $(\tilde f,\tilde h)\in L^q(I,L^2(\Omega))\times L^q(I,H^{1/2}(\partial\Omega))$ measurability of the mapping $t>0\mapsto \pi(t)\in H^1(\Omega)$ follows from Remark~\ref{rem:pressure-operator}. 
Integrating \eqref{est:final} and using regularity of $\dot\U$ we obtain \eqref{final_reg}.
It remains to show that $\U$ is actually the unique weak solution of \eqref{P1.1}--\eqref{P1.6}.
The function $\U$ satisfies $\U\in C([0,+\infty), \HH)\cap L^\infty_{loc}([0,+\infty),\HH)$, $\dot\U\in L^1_{loc}([0,+\infty),\VV^*)$ and the equation \eqref{eq:weak1} holds almost everywhere in $(0,+\infty)$. As $\U\in L^\infty_{loc}([0,+\infty),\HH)$ we have $u_b\in L^2_{loc}([0,\infty),L^2(\Omega))$. The initial values are attained by \cite[Proposition~2.1.1 and Proposition~2.1.4 (i)]{Lunardi1995} and Proposition~\ref{P_dense}. To show that $u$ is a weak solution of \eqref{P1.1}--\eqref{P1.6} it remains to prove $u\in L^2_{loc}([0,+\infty), H^1(\Omega))$. We know $\U\in L^q(I,D(\A))$, $\U\in L^\infty(I, \X_0)$, consequently $u\in L^q(I,H^2(\Omega))$, $u\in L^\infty(I,L^2(\Omega))$ and an interpolation theorem gives 
% napriklad Lemma 2.18 v MaNeRoRu1995
$$
\|u(t)\|^2_{H^1(\Omega)}\leq C\|u(t)\|_{H^2(\Omega)}\|u(t)\|_{L^2(\Omega)}.
% \leq
% C\|u(t)\|_{H^2(\Omega)}e^{\omega t}\|u_0\|_{H^{1/2}(\partial\Omega)}\leq
% C\|u(t)\|_{H^2(\Omega)}e^{\omega t}.
$$
It is enough to integrate this inequality over the time interval to get $u\in L^2_{loc}([0,+\infty), H^1(\Omega))$. This concludes the proof of the regularity properties in the case of assumptions \ref{ass:thmb} or \ref{ass:thmc}.

If the assumption \ref{ass:thma} holds we can proceed similarly. Recall $I=(0,T)$ with $T\in(0,+\infty)$. First we note that from analyticity of $\A$ we get that the mild solution $\U$ defined in \eqref{Sg_solution} satisfies $U\in L^\infty(I,\X_0)$. Then we rewrite the equation \eqref{ACP} as
$$
\partial_t \U = \tilde\A \U + \tilde \F(t),
$$
with $\tilde \A:=\A-\lambda_0$, $\tilde \F(t):=\F(t)+\lambda_0\U$ and $\lambda_0=\max(0,-\alpha/\beta)+1$. The constant $\lambda_0$ is chosen such that $\overline{S_{\theta,0}}$ is a subset of the resolvent set of $\tilde \A$; see Proposition~\ref{propweaksol} and Theorem~\ref{T_reg}.  The function $\tilde\F$ and the semigroup $\tilde\T$ generated by the operator $\tilde\A$ can be estimated
$$
\|\tilde\F\|_{L^q(I,\X_0)}\leq C(\|\F\|_{L^q(I, \X_0)}+\|\U_0\|_{\X_0})\quad\mbox{and}\quad
\exists\omega<0,\forall t>0: \|\tilde\T(t)\|_{\LL(\X_0)}\leq Ce^{\omega t}.
$$
The rest of the proof of regularity can be done as in the cases \ref{ass:thmb} or \ref{ass:thmc}.

It remains to show that the equations \eqref{P1.1}--\eqref{P1.6} hold almost everywhere. It is clear for \eqref{P1.2} a.e. in $I\times\Omega$ and for \eqref{P1.4} a.e. in $I\times\partial\Omega$. We have already identified, under assumptions of the theorem, the mild and the weak solutions, and we reconstructed the pressure $\pi$ so that for $\varphi\in H^1(\Omega)$ with $\varphi_\nu=0$ at $\partial\Omega$ at almost every $t\in I$
$$
 \int_{\Omega}\partial_t u \varphi + \beta\int_{\partial\Omega}\partial_t u \varphi +\int_\Omega(2 Du:D\varphi -\pi\dvg\varphi)+ \alpha\int_{\partial\Omega} u\varphi=
  \int_{\Omega}f\varphi + \beta\int_{\partial\Omega}g\varphi.
$$
Regularity of $u$ and $\pi$ allows us to use the Divergence theorem in the third integral to get for $\varphi\in H^1(\Omega)$ with $\varphi_\nu=0$ at $\partial\Omega$ and $\eta\in{\cal D}(I)$
$$
 \int_I\int_{\Omega}(\partial_t u + \Delta u +\nabla\pi- f) \varphi\eta + \int_I\int_{\partial\Omega} (\beta\partial_t u  + 2 Du \nu + \alpha u - g)\varphi\eta=0.
 $$
 It follows that \eqref{P1.1} must hold a.e. in $I\times\Omega$ and \eqref{P1.3} must hold a.e. in $I\times\partial\Omega$.

 We have already discussed that the initial values are attained in $\HH$, which also means that \eqref{P1.5} holds a.e. in $\{0\}\times\Omega$ and \eqref{P1.6} holds a.e. in $\{0\}\times\partial\Omega$.
\qed

% \section{Appendix}
% \begin{lemma}
%   $\Hpulsg=\gamma(\Hjs)$ and $\Htripulsg=\gamma(\H2s)$.
% \end{lemma}

% \begin{proof}
% $"\supset":$ clear by the Trace Theorem. The zero normal component is inherited.

% $"\subset":$ Fix $\varphi\in \Hpulsg$, there is $u\in H^1(\Omega)$ such that $\gamma(u)=\varphi$ by Reverse Trace Theorem. By Bogovskii theorem, e.g. \cite[Lemma~3.3]{AmGi1994}, we find $v\in H^1_0(\Omega)$ such that $\dvg v=\dvg u$ in $\Omega$. This is possible since
% $$
% \int_\Omega\dvg u=\int_{\partial\Omega}u\cdot\nu=\int_{\partial\Omega}\varphi\cdot\nu=0.
% $$
% Then $u-v\in H^1(\Omega)$, $\dvg (u-v)=0$ and $\gamma(u-v)=\varphi$. Moreover, $u-v\in \L2s$ by \cite[Theorem~2.9]{AmGi1994} and consequently $u-v\in\Hjs$. Similarly, if $\varphi\in\Htripulsgm$, there is $u\in H^2(\Omega)$ such that $\gamma(u)=\varphi$ by Reverse Trace Theorem. 

%\end{proof}

\section*{Acknowledgements}
This study was funded by the Czech Science Foundation project 20-11027X. We thank the anonymous referees for their valuable remarks and suggestions. 

%\section*{Declaration}
%\begin{itemize}
%\item Funding: This study was funded by the Czech Science Foundation project 20-11027X.
%\item Conflict of interest/Competing interests: The authors have no relevant financial or non-financial interests to disclose.
%\end{itemize}

%\section*{Data availability statement}

%Data sharing not applicable to this article as no datasets were generated or analysed during the current study.

%\section*{Conflict of interest declaration}

%The authors declare that there is no conflict of interest.

\bibliographystyle{plain}
\bibliography{vse}

% \begin{thebibliography}{1}

% \bibitem{AmGi1994} Ch. Amrouche, V. Girault, , \emph{Decomposition of vector spaces and application to the Stokes problem in arbitrary dimension}, Czechoslowak Math. J. \textbf{44} (1994), no. 119, 109--140.

% \end{thebibliography}

\end{document}